\documentclass[12pt]{article}
\usepackage{amsmath,amsthm,amscd,amssymb}
\usepackage[mathscr]{eucal}
\usepackage{amssymb}
\usepackage{latexsym}

\newtheorem{Def}{Definition}[section]
\newtheorem{Thm}[Def]{Theorem}
\newtheorem{Prop}[Def]{Proposition}
\newtheorem{Rem}[Def]{Remark}

\newtheorem{Cor}[Def]{Corollary}

\newtheorem{Lem}[Def]{Lemma}

\newtheorem{Assm}[Def]{Assumption}

\numberwithin{equation}{section}
\newcommand{\Q}{\mathbb{Q}}
\newcommand{\R}{\mathbb{R}}
\newcommand{\C}{\mathbb{C}}
\newcommand{\Z}{\mathbb{Z}}
\newcommand{\F}{\mathbb{F}}
\newcommand{\hh}{\mathbb{H}}

\newcommand{\sym}{\mathrm{Sym}}

\newcommand{\trn}[1][1]{{}^t \hspace{-#1pt}}
\newcommand{\mat}[4]{\begin{pmatrix} #1 & #2 \\ #3 & #4 \end{pmatrix}}
\newcommand{\smat}[4]{\left(\begin{smallmatrix} #1 & #2 \\ #3 & #4 \end{smallmatrix}\right)}
\newcommand{\h}{\mathbb{H}}

\begin{document}

\title{On $p$-adic Siegel--Eisenstein series from a point of view of the theory of mod $p^m$ singular forms}
\author{Siegfried B\"ocherer and Toshiyuki Kikuta$^*$}
\maketitle


\begin{abstract}
We show that the $p$-adic Siegel--Eisenstein series of general degree attached to two kind of number sequences are both linear combinations of genus theta series of level dividing $p$, 
by applying the theory of mod $p$-power singular forms. 
As special cases of this result, we derive the results of Nagaoka and Katsurada--Nagaoka. 
\end{abstract}

\section{Introduction and Main results}
Serre \cite{Se} developed the theory of $p$-adic modular forms. 
In the process, he introduced the notion of $p$-adic Eisenstein series, 
defined as the $p$-adic limit of the Eisenstein series, 
and described some of their interesting properties. 
For example, it is stated that if a sequence of weights converges in a suitable sense, 
then all corresponding $p$-adic Eisenstein series coincide with Hecke's Eisenstein series of level $p$.
Subsequently, some generalization to the case of Siegel--Eisenstein series was first given by Nagaoka \cite{Na}, and later given by several people.
We prepare some notation to explain them in more detail.

Let $p$ be an odd prime. 
We put 
\[{\boldsymbol X}=\Z_p\times \Z/(p-1)\Z. \]
We can regard $\Z\subset {\boldsymbol X}$ via the embedding $n\mapsto (n,\widetilde{n})$ with $\widetilde{n}:=n$ mod $p-1$. 
Let $E_{k}^{(n)}$ be the Siegel--Eisenstein series of degree $n$, weight $k$, level $1$ 
and the constant term $1$. 
For $(k,a)\in {\boldsymbol X}$, the $p$-adic Siegel--Eisenstein series $\widetilde{E}^{(n)}_{(k,a)}$
of degree $n$ and weight $(k,a)$ is defined as 
\[\widetilde{E}^{(n)}_{(k,a)}:=\lim _{m\to \infty }E_{k_m}^{(n)} \quad (p\text{-adic\ limit}) \]
where $k_m$ is a number sequence such that $k_m\to \infty $ ($m\to \infty$) in the usual topology of $\mathbb{R}$ 
and $k_m\to (k,a)$ ($m\to \infty$) in ${\boldsymbol X}$. 
In this paper, we mainly discuss the case when the weights are $(k,k)$ with $k$ even and $(k,k+\frac{p-1}{2})$ 
with condition $k\equiv \frac{p-1}{2}$ mod $2$.

We write down  the known results in chronological order:
\begin{itemize}  
\setlength{\itemsep}{-3pt}
\item
Nagaoka \cite{Na} proved that $\widetilde{E}_{(1,\frac{p+1}{2})}^{(n)}$ is just 
the genus theta series for a binary quadratic form $S$ of level $p$ with $\det(2S)=p$. 
\item
Katsurada--Nagaoka \cite{Kat-Na} proved that 
$\widetilde{E}_{(k,k+\frac{p-1}{2})}^{(2)}$ (with $p>2k$) can be written by a linear combination 
of genus theta series of two quadratic forms of matrix size $2k$ and level $p$ 
and a twisted Siegel--Eisenstein series.
\item 
Mizuno \cite{Mizuno2} showed that $\widetilde{E}_{(k,k+\frac{p-1}{2})}^{(2)}$ (with $p>2k$) can be expressed in terms of the Siegel--Eisenstein series of level $p$.
\item
Nagaoka and the second author \cite{Ki-Na} proved that $\widetilde{E}^{(2)}_{(2,2)}$
coincides with the genus theta series for a quaternary quadratic form $S$ of level $p$ with $\det(2S)=p^2$. 
\item
Mizuno--Nagaoka \cite{Miz-Na} showed the modularity (for level $p$) of $\widetilde{E}_{(k,k)}^{(2)}$ (with $p>2k$, $k\ge 2$) by using the Maass lift of Jacobi forms.
\item
Takemori \cite{Ta} proved that $\widetilde{E}^{(2)}_{(k,a)}$ ($k\ge 4$) is a Siegel--Eisenstein
series of level $p$ with character mod $p$ not necessarily quadratic. 
Actually, he studied the $p$-adic Siegel--Eisenstein series with level. 
\item
Katsurada--Nagaoka \cite{Kat-Na2} showed that the coincidence (in \cite{Ki-Na}) of item 4
holds for $\widetilde{E}^{(n)}_{(2,2)}$. 
\end{itemize}
There is another work by Mizuno \cite{Mizuno} for the case with level between \cite{Mizuno2} and \cite{Ki-Na}.
We should mention that the results for $n\ge 4$ in \cite{Na} and for $n\ge 6$ in \cite{Kat-Na2} should impose a condition 
that some sequences of Bernoulli numbers converge to nonzero values $p$-adically (as in Assumption \ref{Assm:NV}); this condition holds when $p$ is a regular prime. 
The need for such conditions has been confirmed by contacting the authors of the papers.
In fact, \cite{Kat-Na2} has been corrected by \cite{Kat-Na3}.

On the other hand, 
in \cite{Bo-Ki,Bo-Ki2}, the authors defined a notion of mod $p^m$ singular forms
and proved that all mod $p^m$ singular forms (with some condition on $p$-rank) 
can be expressed by linear combinations of theta series.  
In this paper, we study the $p$-adic Siegel--Eisenstein series from a point of view 
of this theory. 

We put ${\boldsymbol k}_0:=(k,k)$ and ${\boldsymbol k}_1:=(k,k+\frac{p-1}{2})$ and denote by $k_j(m)$ 
a number sequence satisfying $k_j(m)\to {\boldsymbol k}_j$ ($m\to \infty$) in ${\boldsymbol X}$.  
Looking at the papers \cite{Na,Kat-Na2}, 
the Fourier coefficients of $\widetilde{E}_{(1,\frac{p+1}{2})}^{(n)}$ and $\widetilde{E}_{(2,2)}^{(n)}$ for the higher ranks are zero. 
Their calculations were for $k=1$, $2$, but the same phenomenon seems to occur for general $k$.
This means, in other words, that some constant multiples of $E_{k_j(m)}^{(n)}$ ($j=0$, $1$) are mod $p^{c(m)}$ singular modular forms for some $c(m)$ satisfying $c(m)\to \infty $ if $m\to \infty$.
Therefore, our structure theorem can be applied to prove the coincidence of the $p$-adic Siegel--Eisenstein series and a linear combination of theta series.
By looking somewhat more precisely, 
we can prove that their levels are $p$ and also that
it is actually a linear combination of genus theta series.

We state our main theorem precisely. 
For simplicity, we assume that $p$ is a regular prime, 
but the actual condition we need is that of Assumption \ref{Assm:NV}. 
Such assumption is always satisfied for regular primes.
\begin{Thm}
\label{thm:main}
Let $n$, $k$ be positive integers and $p$ a regular prime with $p>2k+1$.
Assume that $k$ is even in ${\boldsymbol k}_0$ and 
$k\equiv \frac{p-1}{2}$ mod $2$ in ${\boldsymbol k}_1$. 
Then we have
\begin{align}
\label{eq:lin0}
\widetilde{E}_{{\boldsymbol k}_j}^{(n)}= \sum _{\substack{{\rm gen}(S)\\ \ {\rm level}(S)\mid p \\ \chi _S=\chi _p^j}}a({\rm gen}(S))\cdot (\Theta ^{(n)}_{{\rm gen}(S)})^0,\quad (a({\rm gen}(S))\in \Q), 
\end{align}
where the summation in (\ref{eq:lin0}) goes over all genera ${\rm gen}(S)$ of positive definite quadratic forms $S$ of matrix size $2k$ with level dividing $p$ and $\chi_S=\chi _p^j$ 
(and therefore $\det (2S)=p^s$ with $s\equiv j$ mod $2$ and $0\le s\le 2k$. 
In such formulation, not all $p^s$ must occur).  
Here, the notation is as follows. 
\begin{itemize}  \setlength{\itemsep}{-3pt}
\item $(\Theta ^{(n)}_{{\rm gen}(S)})^0$ is the unnormalized genus theta series defined by 
\[(\Theta ^{(n)}_{{\rm gen}(S)})^0:=\sum _{S\in {\rm gen}(S)}\frac{1}{\epsilon (S)}\theta ^{(n)}_S, \] 
\item
$\theta ^{(n)}_S$ is the usual theta series of degree $n$ attached to $S$ (see (\ref{eq:theta}) in Subsection \ref{subsec:theta}),
\item
$\epsilon (S)$ is the order of automorphism group of $S$, 
\item
$\chi _S$ is the character associated with $S$ (see (\ref{eq:char}) in Subsection \ref{subsec:theta}), 
\item
$\chi _p:=(\frac{*}{p})$ is the unique nontrivial quadratic character mod $p$, and we understand $\chi _p^0$ as the trivial character mod $p$. 
\end{itemize}
\end{Thm}
\begin{Rem}
\begin{enumerate} \setlength{\itemsep}{-3pt}
\item
We can prove that the $p$-adic Siegel--Eisenstein series $\widetilde{E}_{{\boldsymbol k}_j}^{(n)}$ does not depend on the choice of number sequence $k_j(m)$
satisfying $k_j(m)\to {\boldsymbol k}_j$ ($m\to \infty $) in ${\boldsymbol X}$. 

\item
We will see that $a({\rm gen}(S))$ is obtained explicitly by taking $p$-adic limit of 
a certain primitive Fourier coefficient of the Siegel--Eisenstein series (see Subsection \ref{local_part}). 

\item
The proof of Theorem \ref{thm:main} will use the result that the mod $p^m$ singular form, introduced later as Theorem \ref{thm:structure}, is represented by a linear combination of theta series. 
The condition $p>2k+1$ is necessary for the use of this theorem.
\item
We compare our result with those of \cite{Na,Kat-Na2}. 
All of these results are for the general degree case. 
Our result is for the case where the weight $k$ is general, 
whereas their results are for $k=1$ in \cite{Na} and $k=2$ in \cite{Kat-Na2}. 
Our result requires the regularity condition on $p$ for all $n\ge 1$. 
Their results require such the condition in the case $n>2k+1$, but not in the case $n\le 2k+1$. 
\item
Our result shows (for regular $p$) that 
the twisted Siegel--Eisenstein series part for 
$\widetilde{E}_{(k,k+\frac{p-1}{2})}^{(2)}$ (with $p>2k+1$) obtained in Katsurada--Nagaoka \cite{Kat-Na} is also
a linear combination of genus theta series. 
Moreover, the description is explicitly given, because of (2).
However, this is also implied from Katsurada--Nagaoka \cite{Kat-Na} with Katsurada--Schulze-Pillot \cite{Kat-Sch}
or also from Mizuno \cite{Mizuno2} with Katsurada--Schulze-Pillot \cite{Kat-Sch}.   
\end{enumerate}
\end{Rem}

By Siegel's main theorem, we know that all genus theta series are elements of the space of the Siegel--Eisenstein series with corresponding level (see e.g. Freitag \cite{Frei}).
Therefore, our $p$-adic Siegel--Eisenstein series   
$\widetilde{E}_{{\boldsymbol k}_j}^{(n)}$
is an element of the space of Siegel--Eisenstein series of level $p$ by our theorem. 
Furthermore, we can prove that the eigenvalue of $\widetilde{E}_{{\boldsymbol k}_j}^{(n)}$ is $1$ for the Hecke operator $U(p)$ (for the definition of $U(p)$, see Subsection \ref{eq:Up}). 
\begin{Cor}
\label{cor:eisen}
In the same situation as Theorem \ref{thm:main}, we have the following statements. 
\begin{enumerate} \setlength{\itemsep}{-3pt}
\item
The $p$-adic Siegel Eisnstein series $\widetilde{E}^{(n)}_{{\boldsymbol k}_j}$ ($j=0$, $1$) is an element of the space of Siegel--Eisenstein series of level $p$ with character $\chi _p^j$. 
\item 
We have 
$\widetilde{E}_{{\boldsymbol k}_j}^{(n)}\mid U(p)=\widetilde{E}_{{\boldsymbol k}_j}^{(n)}$.
\end{enumerate}
\end{Cor}
\begin{Rem}
\label{rem:multi}
\begin{enumerate}
\setlength{\itemsep}{-3pt}
\item 
The statement (1) for $\widetilde{E}^{(2)}_{{\boldsymbol k}_1}$ is proved directly by Mizuno \cite{Mizuno2}. 
\item
The results of Katsurada--Nagaoka \cite{Kat-Na,Kat-Na2} also yield the claims (1) for $\widetilde{E}^{(2)}_{{\boldsymbol k}_1}$ and $\widetilde{E}^{(n)}_{(2,2)}$. 
\item 
If $k>(n+1)/2$, 
the statement (2) completely characterizes the Siegel--Eisenstein series up to constant multiples,  since the multiplicity is one for $U(p)$-eigenvalue $1$ (see Gunji \cite[Corollarly 2.5]{Gu}).
If $k$ is small, the situation is more delicate for two reasons:
One has to look at Siegel--Eisenstein series of level $p$ 
after analytic continuation. 
One only knows that the multiplicity of the $U(p)$-eigenvalue $1$ is less or equal
to two on this space of Siegel--Eisenstein series \cite{Gu}.
Furthermore, one needs Siegel's main theorem outside the range of convergence.
Such a version is indeed available and one can deduce the explicit classical version
fitting into our context from \cite[(4.44)]{Kudla}, as was kindly pointed out to us by Schulze-Pillot \cite{Schupi}. 
\end{enumerate}
\end{Rem}

Finally, we mention the advantages of our method:
The strength of our approach is that we can check
the equality (\ref{eq:lin0})
without hard calculations of local densities as they
do in \cite{Ki-Na,Kat-Na2}. 
This method allows us also to give generalizations of their results in terms of weight $k$ and degree $n$.

\section{Preliminaries}
\label{Sec:2}
\subsection{Siegel modular forms}
\label{sec:siegel-modular-forms}
Let $n$ be a positive integer and
$\hh_{n}$ the Siegel upper half space of degree $n$ defined as
\begin{equation*}
  \h_{n}:=\left\{X+ i Y \; | \;
    X,\ Y\in \sym_{n}(\R), \ Y>0
  \right\},
\end{equation*}
where $Y>0$ means that $Y$ is positive definite, and $\sym_{n}(R)$ is the set of symmetric matrices of size $n$ with components in $R$.

The Siegel modular group $\Gamma _n$ of degree $n$ is defined by 
\begin{equation*}
  \Gamma_{n} := \left\{\gamma \in \mathrm{GL}_{2n}(\Z)
    \; |\; \trn \gamma J_{n} \gamma = J_{n}
  \right\},
\end{equation*}
where $J_{n} = \smat{0_{n}}{-1_{n}}{1_{n}}{0_{n}}$ and $0_{n}$ (resp. $1_{n}$)
is the zero matrix (resp. the identify matrix) of size $n\times n$.

Let $N$ be a positive integer. 
In this paper, 
we deal with the congruence subgroup $\Gamma _0^{(n)}(N)$ of $\Gamma _n$ defined as 
\begin{align*}
&\Gamma _0^{(n)}(N):=\left\{ \begin{pmatrix}A & B \\ C & D \end{pmatrix}\in \Gamma _n \: \Big| \: C\equiv 0_n \bmod{N} \right\}.
\end{align*}
Here $A$, $B$, $C$, $D$ are $n \times n$ matrices.

We define an action of $\Gamma _n$ on $\hh_{n}$ by
$\gamma Z := (AZ + B)(CZ + D)^{-1}$ for $Z \in \hh_{n}$, $\gamma = \left( \begin{smallmatrix} A & B \\ C & D \end{smallmatrix}\right) \in \Gamma _n$.
For a holomorphic function $F:\mathbb{H}_n\longrightarrow \mathbb{C}$ and a matrix $\gamma =\left( \begin{smallmatrix} A & B \\ C & D \end{smallmatrix}\right)\in \Gamma _n$,
we define a slash operator by
\[F|_k\; \gamma :=\det(CZ+D)^{-k}F(\gamma Z).\]

For a positive integer $k$ and a Dirichlet character 
$\chi $ mod $N$, 
the space $M_k(\Gamma _0^{(n)}(N),\chi )$ of Siegel modular forms of weight $k$ with character $\chi$ consists of all of holomorphic functions $F:\mathbb{H}_n\rightarrow \mathbb{C}$ satisfying
\begin{equation*}
F|_{k}\: \gamma =\chi (\det D)F(Z)\quad \text{for}\quad \gamma =\begin{pmatrix}A & B \\ C & D \end{pmatrix}\in \Gamma _0^{(n)}(N).
\end{equation*}
If $n=1$, the usual condition for the cusps should be added. 
When $\chi = {\boldsymbol 1}_N$ (trivial character mod $N$), 
we write simply $M_k(\Gamma _0^{(n)}(N))$ for $M_k(\Gamma _0^{(n)}(N),{\boldsymbol 1}_N)$.

Any $F \in M_k(\Gamma _0^{(n)}(N), \chi )$ has a Fourier expansion of the form
\[
F(Z)=\sum_{0\leq T\in\Lambda_n}a_F(T)q^T,\quad q^T:=e^{2\pi i {\rm tr}(TZ)},
\quad Z\in\mathbb{H}_n,
\]
where
\[
\Lambda_n
:=\{ T=(t_{ij})\in {\rm Sym}_n(\mathbb{Q})\;|\; t_{ii},\;2t_{ij}\in\mathbb{Z}\}.
\]

For a subring $R$ of $\mathbb{C}$, we denote by $M_{k}(\Gamma _0^{(n)}(N), \chi )_{R}$ 
the $R$-module consisting of all $F\in M_{k}(\Gamma _0^{(n)}(N), \chi )$ satisfying $a_F(T)\in R$ for all $T\in \Lambda _n$. 
We write simply $M_{k}(\Gamma _0^{(n)}(N))_{R}$ for  
$M_{k}(\Gamma _0^{(n)}(N), {\boldsymbol 1}_N)_{R}$.  
\subsection{Theta series for quadratic forms}
\label{subsec:theta}
For two matrices $A$, $B$, we write $A[B]:={}^tBAB$ when the products are defined. 
Let $m$ be a positive integer. 
For $S$, $T\in \Lambda _m$, we write $S\sim T$ mod ${\rm GL}_m(\Z)$ if 
there exists $U\in {\rm GL}_m(\Z)$ such that $S[U]=T$.
We say that $S$ and $T$ are ``${\rm GL}_m(\Z)$-equivalent'' if $S\sim T$ mod ${\rm GL}_m(\Z)$. 
We denote by $\Lambda _m^+$ the set of all positive definite elements of $\Lambda _m$. 
We put $L:=\Lambda _m$ or $\Lambda ^+_m$. 
We write $L/{\rm GL}_m(\Z)$ for $L/\sim $  
the set of representatives of ${\rm GL}_m(\Z)$-equivalence classes in $L$.   

Let $m$ be even. 
For $S\in \Lambda _m^+$, we define the theta series of degree $n$ in the usual way:
\begin{align}
\label{eq:theta}
\theta _S^{(n)}(Z):=\sum _{X\in \Z^{m,n}}e^{2\pi i({\rm tr}(S[X]Z))}\quad (Z\in \hh_{n}), 
\end{align}
where $\Z^{m,n}$ is the set of  $m\times n$ matrices with integral components. 
We define the level of $S$ as 
\[{\rm level}(S):=\min\{N\in \Z_{\ge 1} \;|\; N(2S)^{-1}\in 2\Lambda _m\}. \]
Then $\theta _S^{(n)}$ defines an element of $M_{\frac{m}{2}}(\Gamma _0^{(n)}(N),\chi _S)$, 
where $N={\rm level}(S)$, $\chi _S$ is a Dirichlet character mod $N$ defined by 
\begin{align}
\label{eq:char}
\chi _S(d)={\rm sign} (d)^\frac{m}{2} \left( \frac{(-1)^\frac{m}{2}\det 2S}{|d|} \right).
\end{align}

For fixed $S\in \Lambda^+ _{m}$ and $T\in \Lambda _n$, we put 
\[A(S,T):=\sharp\{X \in \Z^{m,n} \;|\; S[X]=T \}.\]
Using this notation, we can write the Fourier expansion of the theta series in the
form 
\[\theta _S^{(n)}(Z)=\sum _{T\in \Lambda _n}A(S,T)q^T. \]

Let $S\in \Lambda _m^+$ and ${\rm gen}(S)$ be the genus containing $S$. 
Let $\{S_1,\cdots ,S_h\}$ be a set representatives of ${\rm GL}_m(\Z)$-equivalence classes in ${\rm gen}(S)$. 
The genus theta series associated with  $S$ is defined by 
\[\Theta ^{(n)}_{{\rm gen}(S)}(Z):=\left(\sum _{i=1}^h\frac{\theta _{S_i}^{(n)}(Z)}{\epsilon(S_i)}\right)/\left(\sum _{i=1}^h\frac{1}{\epsilon (S_i)}\right), \]
where we put $\epsilon (T):=A(T,T)$ for $T\in \Lambda_m^+$. 
Then we have also $\Theta ^{(n)}_{{\rm gen}(S)}\in M_{\frac{m}{2}}(\Gamma _0^{(n)}(N),\chi _S)$
with $N={\rm level}(S)$. 
For $(\Theta ^{(n)}_{{\rm gen}(S)})^0$ defined in Introduction, we have the following relation  
\[(\Theta ^{(n)}_{{\rm gen}(S)})^0=\left(\sum _{i=1}^h\frac{1}{\epsilon (S_i)}\right)\cdot \Theta ^{(n)}_{{\rm gen}(S)}. \]

\subsection{Siegel--Eisenstein series}
We put 
\[\Gamma _\infty ^{(n)}:=\left\{\mat{A}{B}{C}{D}\in \Gamma _n\; \Big| \; C=0_n\right\}. \]
For an even integer $k$ with $k>n+1$, the Siegel--Eisenstein series of degree $n$ and weight $k$ with level $1$ is defined as 
\[E_k^{(n)}(Z):=\sum _{\smat{*}{*}{C}{D}\in \Gamma _\infty^{(n)}\backslash \Gamma_n}\det (CZ+D)^{-k},\quad Z\in \hh_n.  \]
This series defines an element of $M_k(\Gamma _n)_\Q$. 
We denote by
\[E_k^{(n)}=\sum _{T\in \Lambda _n }a_k^{(n)}(T)q^T\]
the Fourier expansion of $E_k^{(n)}$.
The formula for $a^{(n)}_k(T)$ is given explicitly by Katsurada \cite{Kat}. 
However, in this paper, we do not require the explicit formulas for all Fourier coefficients. 
The results from \cite{Boe} on integrality properties of Fourier coefficients and primitive Fourier coefficients are sufficient for our purpose. 
We review these results below.

We put $\Lambda _n^{(r)}:=\{T\in \Lambda _n\;|\; {\rm rank}(T)=r\}$. 
Note that, if $T\in \Lambda ^{(r)}_n$ then we have
 $a_k^{(n)}(T)=a_k^{(r)}(T')$ for $T'\in \Lambda _r^+$ such that $T\sim \smat{T'}{0}{0}{0}$ mod ${\rm GL}_n(\Z)$.
Therefore, we may assume that $T\in \Lambda _{r}^+$ and may consider $a_k^{(r)}(T)$, 
instead of $a_k^{(n)}(T)$ with $T\in \Lambda _n^{(r)}$. 
\begin{Thm}[\cite{Boe}]
\label{thm:int}
We put  
\begin{align*}
c_{k,n}:=
\begin{cases} 
\displaystyle 2^n \cdot \frac{k}{B_k} \cdot \prod_{i=1}^{\frac{n-1}{2}}\frac{k-i}{B_{2k-2i}}\quad \text{if}\quad n\ \text{odd}, \\
\displaystyle 2^n \cdot \frac{k}{B_k}\cdot \frac{1}{D_{2k-n}^*}\cdot \prod_{i=1}^{\frac{n}{2}}\frac{k-i}{B_{2k-2i}}\quad \text{if}\quad n\equiv 0\bmod{4}, \\
\displaystyle 2^{n-1}\cdot \frac{k}{B_k}\cdot \frac{1}{D^{**}_{2k-n}}\cdot \prod_{i=1}^{\frac{n}{2}}\frac{k-i}{B_{2k-2i}}\quad \text{if}\quad n\equiv 2 \bmod{4},
\end{cases}
\end{align*}
where $D^*_{2k-n}$, $D_{2k-n}^{**}$ are two natural numbers defined as 
\begin{align*}
D_{2k-n}^*&:=\prod_{p|D_{2k-n}}p^{1+v_p(k-\frac{n}{2})}\quad for\ n\equiv 0 \bmod{4},\\
D_{2k-n}^{**}&:=\prod_{\substack{p|D_{2k-n} \\ p\equiv3 \bmod{4}}}p^{1+v_p(k-\frac{n}{2})}\quad for\ n\equiv 2 \bmod{4},
\end{align*}
where $D_{2k-n}$ is the denominator of the $(2k-n)$-th Bernoulli number $B_{2k-n}$. 
Then we have $a_{k}^{(n)}(T)\in c_{k,n} \mathbb{Z}$ for $T\in \Lambda _{n}^+$.  
\end{Thm}
Let $a_k^{(n)}(T)^*$ be the primitive Fourier coefficient of Siegel--Eisenstein series defined by 
\begin{align}
\label{eq:rel_pri}
a_k^{(n)}(T)=\sum _{D}a_k^{(n)}(T[D^{-1}])^*, 
\end{align}
where $D$ runs over all elements of ${\rm GL}_n(\Z)\backslash \{D\in \Z^{n,n}\; |\; \det D\neq 0\}$
satisfying $T[D^{-1}]\in \Lambda _n^+$. 
Note that the sum is actually finite. 
Using Hecke theory for ${\rm GL}_n$, one can show that such primitive Fourier coefficients exist and can be explicitly written as finite
integral linear combination of usual Fourier coefficients.  
We do not need this explicit expression in this paper.
However, both these relations depend on the global properties of $T$, but $a_k^{(n)}(T)$ and 
$a_k^{(n)}(T)^*$ are given by local properties (of different type).

\begin{Thm}[\cite{Boe}]
\label{thm:pri}
Let $n$ be a positive even integer. 
For any $T\in \Lambda _{n}^+$, we can separate $a_k^{(n)}(T)^*$ into a ``Bernoulli part'' and a ``local part'' 
\[a_k^{(n)}(T)^*=b_k^{(n)}(T)\cdot c_k^{(n)}(T), \]
where 
\begin{align*}
b_k^{(n)}(T)
&=2^{n-\frac{n}{2}}\cdot \frac{k}{B_k}\cdot \frac{B_{k-\frac{n}{2},\eta _T}}{k-\frac{n}{2}} \cdot \prod _{i=1}^{\frac{n}{2}}\frac{2k-2i}{B_{2k-2i}} 
\end{align*}
and $c_k^{(n)}(T)$ being always integral, and (for even $n$) it is given explicitly by  
\begin{align*}
c_k^{(n)}(T)= (\det (2T) f_T^{-1})^{k-\frac{n+1}{2}}\prod _{q\mid \det(2T)}(1-\eta _T(q)q^{\frac{n}{2}-k})B^*_q(J^{2k},T), 
\end{align*}
with 
\[
B^*_q(J^{2k},T)
=\begin{cases}
&\prod _{j=1}^{\frac{s-1}{2}}(1-q^{2j+n-2k}) \quad \text{if}\quad s\ \text{odd},\\
&(1+\lambda _q(T)q^{\frac{n+s-2k}{2}})\prod_{j=1}^{\frac{s}{2}-1}(1-q^{2j+n-2k}) \quad \text{if}\quad s\ \text{even}. 
\end{cases}
\]
In these formulas, 
\begin{itemize}  \setlength{\itemsep}{-3pt}
\item 
$s$ is defined as the integer for which the rank over $\F_q$ of $T$ is $n-s$, 
\item
$\eta _T$ is the primitive character associated with $\chi _T$, 
\item
$\lambda _q(T)$ is the quadratic character attached to the regular part of $T$ over $\F_q$, i.e., 
$\lambda _q(T):=\chi _S(q)$ when $T[U]=\smat{S}{0}{0}{0_s}$ with $\det S\neq 0$ in $\F_q$ for some $U\in {\rm GL}_n(\F_q)$.  
Here we regard as $\lambda _q(T)=1$ if $s=n$. 
\end{itemize}
\end{Thm}
\begin{Rem}
\begin{enumerate}  \setlength{\itemsep}{-3pt}
\item
As was pointed out in \cite[page 795]{Koh}, there are some misprints in \cite{Boe},
which luckily cancel each other in the final formula. We may therefore completely ignore them here.
\item 
In our arguments, we will need the above property only for even $n$.
\end{enumerate}
\end{Rem}

\subsection{Congruences for modular forms}
Let $p$ be an odd prime and $v_p$ the additive valuation on the $p$-adic field $\Q_p$ normalized such that $v_p(p)=1$. 
For a formal power series $F$ of the form $F=\sum _{T\in \Lambda _{n}}a_{F}(T)q^T$ with $a_{F}(T)\in \Q_p$,  
we define
\begin{align*}
&v_p(F):=\inf \{v_p(a_{F}(T))\; |\; T\in \Lambda _n\}, \\
&v_p^{(r)}(F):=\inf \{v_p(a_{F}(T))\; |\; T\in \Lambda _n,\ {\rm rank}(T)=r\}.  
\end{align*}
\begin{Thm}[\cite{Bo-Ki}]
\label{thm:wt_sing}
Let $k$ be a positive integer.  
Let $F\in M_k(\Gamma _0^{(n)}(N),\chi )_{\Q}$ with 
a quadratic Dirichlet character $\chi$ mod $N$. 
Then the equality 
$v_p^{(r+1)}(F)=v_p^{(r)}(f)+m$ ($r<n$) with $m\geq 1$ is 
possible only if $2k-r\equiv 0$ mod $(p-1)p^{m-1}$ and in particular $r$ is even.    
\end{Thm}
Using this theorem, we get properties of the weights of 
``mod $p^m$ singular'' modular forms defined as follows. 
Let $\Z_{(p)}$ be the set of $p$-integral rational numbers. 
\begin{Def}
We say that $F\in M_k(\Gamma _0^{(n)}(N),\chi )_{\Z_{(p)}}$ is ``mod $p^m$ singular'' if there exists $r$ satisfying the following properties: 
\begin{itemize}  \setlength{\itemsep}{-3pt}
\item
$a_F(T)\equiv 0$ mod $p^m$ for any $T\in \Lambda _{n}$ with ${\rm rank}(T)>r$, 
\item
there exists $T\in \Lambda _n$ with ${\rm rank}(T)=r$ such that $a_F(T)\not \equiv 0$ mod $p$. 
\end{itemize} 
We call such $r$ ``$p$-rank'' of $F$. 
\end{Def}
\begin{Cor}[\cite{Bo-Ki}]
Let $p$ be an odd prime and $k$ a positive integer.  
Let $F\in M_k(\Gamma _0^{(n)}(N),\chi )_{\Z_{(p)}}$ with 
$\chi $ a quadratic Dirichlet character mod $N$. 
Suppose that $F$ is mod $p^m$ singular of $p$-rank $r$. 
Then we have $2k-r\equiv 0$ mod $(p-1)p^{m-1}$.
In particular, $r$ should be even.      
\end{Cor}

Let $F_i$ ($i=1$, $2$) be two formal power series of the form
\[F_i=\sum _{T\in \Lambda _{n}}a_{F_i}(T)q^T\]
with $a_{F_i}(T)\in \Z_{(p)}$ for all $T\in \Lambda _n$. 
We write $F_1 \equiv F_2$ mod $p^m$ if $a_{F_1}(T)\equiv a_{F_2}(T)$ mod $p^m$ for all $T \in \Lambda _n$.  

We have now the following structure theorem.  
\begin{Thm}[\cite{Bo-Ki2}]
\label{thm:structure}
Let $n$ be a positive integer and $r$ an even integer with $n\ge 2r$. 
Let $p$ be an odd prime with $p>r+1$.
Suppose that $F\in M_{k}(\Gamma _{n})_{\Z_{(p)}}$ is mod $p^m$ singular with $p$-rank $r$.  
Then there exists $e\ge 0$ such that 
\[F\equiv \sum _{\substack{S\in \Lambda _r^+/{\rm GL}_r(\Z)\\ {{\rm level}(S)}\mid p^e
\\ \chi_S=\chi _p ^t }} c_S \theta _S^{(n)} \bmod{p^m} \quad (c_S\in \Z_{(p)}), \]
where $t$ is an integer satisfying $k-\frac{r}{2}=t\cdot \frac{p-1}{2}\cdot p^{m-1}$. 
\end{Thm}

\begin{Def}
Let $G$ be a formal power series of the form
\[G=\sum _{T\in \Lambda _{n}}a_{G}(T)q^T\]
with $a_{G}(T)\in \Q_p$ for all $T\in \Lambda _n$. 
Let $k_m$ be a number sequence of positive even integers such that $k_m\to \infty $ ($m\to \infty$) in the usual topology of $\mathbb{R}$.
For the sequence of the Siegel--Eisenstein series $E_{k_m}^{(n)}$ of weight $k_m$, we write 
\begin{align*}
\lim _{ m\to \infty} E_{k_m}^{(n)}=G\quad (p\text{-adic\ limit})
\end{align*}
if $v_p(E_{k_m}^{(n)}-G)\to \infty $ when $m\to \infty $.
Then we say that $G$ is a ``$p$-adic Siegel--Eisenstein series''. 
If $k_m\to (k,a)$ ($m\to \infty $) in ${\boldsymbol X}$, then we call $(k,a)$ the ``weight'' of $G$ and we write $G=\widetilde{E}_{(k,a)}^{(n)}$.  
\end{Def}
\section{Proof of the main theorem}
As in Introduction, we put ${\boldsymbol k}_0:=(k,k)$ with $k$ even and ${\boldsymbol k}_1:=(k,k+\frac{p-1}{2})$ with condition $k\equiv \frac{p-1}{2}$ mod $2$. 
For any $k_j(m)$ satisfying $k_j(m)\to {\boldsymbol k}_j$ ($m\to \infty $) in ${\boldsymbol X}$, 
we may assume that $k_j(m)$ is of the form
\[k_j(m)=k+a_j(m)p^{b(m)}, \]
where $a_j(m)$ is a positive integer with $a_j(m)\equiv \frac{p-1}{2^j}$ mod $p-1$, 
and $b(m)=b_j(m)$ is a positive integer satisfying $b(m)\to \infty $ if $m\to \infty$.  
Thereafter, whenever ${\boldsymbol k}_j$ or $k_j(m)$ is involved, 
the above condition on $k$ is always assumed.

For simplicity, we assume that $p$ is a regular prime when considering the $p$-adic Siegel--Eisenstein series $\widetilde{E}^{(n)}_{{\boldsymbol k}_j}$, 
but the actual condition we need is the following assumption. 
\begin{Assm}
\label{Assm:NV}
We suppose that 
\begin{align}
\label{eq:NVB}
\lim _{m\to \infty }\frac{B_{2k_j(m)-2i}}{2k_j(m)-2i}\neq 0 \quad (p\text{-adic\ limit})
\end{align}
for all $i$ with $k<i\le [n'/2]$ and $i\not \equiv k$ mod $\frac{p-1}{2}$.
Here $n'=n$ if $n>4k$ and $n'=4k$ if $n\le 4k$.  

This means, using the terminology of the $p$-adic $L$-function, 
\begin{align}
\label{eq:NVL}
(\lim _{m\to \infty }L_p(1-2k_j(m)+2i,\omega ^{2k-2i})=)\ L_p(1-2k+2i,\omega ^{2k-2i})\neq 0
\end{align}
for all $i$ with $k<i\le [n'/2]$ and $i\not \equiv k$ mod $\frac{p-1}{2}$. 
Here $L_p(s,\chi )$ is the Kubota--Leopoldt $p$-adic $L$-function and $\omega $ is the Teichm\"uller character. 
\end{Assm}
\begin{Rem}
\begin{enumerate} \setlength{\itemsep}{-3pt}
\item
The number sequence $\left\{\frac{B_{2k_j(m)-2i}}{2k_j(m)-2i}\right\}$ as above always converges to some number in $\Q_p$, 
because it is a Cauchy sequence. 
\item
If $p$ is a regular prime, then $p\nmid \frac{B_k}{k}$ for all even $k\ge 2$ with $k\not \equiv 0$ mod $p-1$ and therefore the condition (\ref{eq:NVB}) (and also (\ref{eq:NVL}))
holds for all $i>k$.  
\item
A condition $L_p(1+2i,\omega ^{-2i})\neq 0$ for all $i\ge 1$ is equivalent to the prediction called the ``higher version of the Leopoldt conjecture'' and 
seems to be a standard conjecture in Iwasawa theory (cf. \cite[Examples 2.19]{Fru}, \cite[Remark 3.2 ii)]{Kol}).
Again, the fact that $p$ is regular is a sufficient condition for the above (equivalence) condition to hold.
\end{enumerate}
\end{Rem}

\subsection{Siegel--Eisenstein series mod $p$-power}
\begin{Prop}
\label{prop:cong}
Let $n$ be a positive integer with $n\ge 4k$ and $p$ a regular prime with $p>2k+1$. 
\begin{enumerate}  \setlength{\itemsep}{-3pt}
\item
Let $2k<r\le n$. 
Then there exists a constant $C_r$ such that 
\[v^{(r)}_p(E^{(n)}_{k_j(m)})\ge C_r+b(m)\] 
holds for sufficiently large $m$.
\item
If we put $\nu _m:=v^{(2k)}_p(E^{(n)}_{k_j(m)})$, then $v ^{(r)}_p(E^{(n)}_{k_j(m)})\ge \nu _m $ for all $r$ with $0\le r< 2k$. 
\item
We have $\nu_m\le 0$ for all $m$. 
\end{enumerate}

In particular, each series $p^{-\nu_m} E^{(n)}_{k_j(m)}$ ($j=0$, $1$) is mod $p^{c(m)}$ singular with $p$-rank $2k$ for some $c(m)$ satisfying $c(m)\to \infty$ if $m\to \infty $. 
\end{Prop}
\begin{Rem}
This proposition holds for $n>2k$. 
The properties (2), (3) hold even without regularity condition on $p$. 
When $n=2k+1$, the regularity condition on $p$ in (1) also is not necessary.
\end{Rem}

\begin{proof}
(1) Recall from Theorem \ref{thm:int} that 
\begin{align*}
a^{(r)}_{k_j(m)}(T)=\frac{k_j(m)}{B_{k_j(m)}}\cdot &\prod _{i=1}^{[\frac{r}{2}]}\frac{2k_j(m)-2i}{B_{2{k_j(m)}-2i}}\times \text{``some\ integer''}\\
&\times
\begin{cases} 
&1 \quad \text{if}\quad r\ \text{odd}\\
&\frac{1}{D^\bullet _{2{k_j(m)}-r}} \quad \text{if}\quad r\ \text{even}, 
\end{cases} 
\end{align*}
where ``$\bullet $'' is ``$*$'' or ``$**$'' according as $r\equiv 0$ mod $4$ or $r\equiv 2$ mod $4$. 

For the factor $\frac{k_j(m)}{B_{k_j(m)}}$, we have
\begin{align*}
&\lim _{m\to \infty } \frac{k_0(m)}{B_{k_0(m)}}=\frac{k}{(1-p^{k-1})B_{k}},\quad \lim _{m\to \infty}\frac{k_1(m)}{B_{k_1(m)}}= \frac{k}{B_{k,\chi _{p}}}. 
\end{align*}
The first formula can be obtained by Kummer's congruence. 
The second formula can be obtained by e.g. using properties of the Kubota--Leopoldt $p$-adic $L$-function (or see Serre \cite{Se}). 
In particular, $v_p \left(\frac{k_j(m)}{B_{k_j(m)}}\right)$ is a constant for sufficiently large $m$.

We consider the factor $\prod_{i=1}^{[r/2]} \frac{2k_j(m)-2i}{B_{2{k_j(m)}-2i}}$. 
Since $r>2k$, the case $i=k$ always appears in this product. 
Therefore we write as 
\[\prod_{i=1}^{[\frac{r}{2}]} \frac{2k_j(m)-2i}{B_{2{k_j(m)}-2i}}=\frac{2k_j(m)-2k}{B_{2{k_j(m)}-2k}}\cdot \prod_{i\neq k}\frac{2k_j(m)-2i}{B_{2{k_j(m)}-2i}}.\]
For the factor $\frac{2k_j(m)-2k}{B_{2{k_j(m)}-2k}}$, the von Staudt--Clausen theorem implies 
\begin{align*}
v_p \left(\frac{2k_j(m)-2k}{B_{2k_j(m)-2k}}\right)=v_p\left(\frac{2a_j(m)p^{b(m)}}{B_{2a_j(m)p^{b(m)}}}\right) \ge b(m)+1. 
\end{align*}
We need to prove that there exists a constant $C'_r$ such that $v_p\left(\prod_{\substack{i\neq k}}\frac{2k_j(m)-2i}{B_{2{k_j(m)}-2i}}\right)\ge C'_r$
holds for sufficiently large $m$.   

Let $i\not \equiv k$ mod $\frac{p-1}{2}$.  
If $i<k$, then Kummer's congruence yields 
\begin{align*}
\lim _{m\to \infty}\frac{2k_j(m)-2i}{B_{2k_j(m)-2i}}=\frac{2k-2i}{(1-p^{2k-2i-1})B_{2k-2i}}. 
\end{align*}
If $i>k$, the condition that $p$ is a regular prime is required here. 
This assumption implies the convergence of $\left\{\frac{2k_j(m)-2i}{B_{2k_j(m)-2i}}\right\}$. 
It follows from these arguments that $v_p\left(\frac{2k_j(m)-2i}{B_{2k_j(m)-2i}}\right)$ are constants for all $i$ with $i\not \equiv k$ mod $\frac{p-1}{2}$ when $m$ is sufficiently large. 

Let $i\equiv k$ mod $\frac{p-1}{2}$.  
The von Staudt--Clausen theorem implies 
\[v_p \left(\frac{2k_j(m)-2i}{B_{2k_j(m)-2i}}\right)\ge 1. \]
Therefore we get some constant $C'_r$ such that $v_p\left(\prod_{i\neq k}\frac{2k_j(m)-2i}{B_{2{k_j(m)}-2i}}\right)\ge C'_r$.

When $r$ is even, we must consider the factor $1/D^\bullet_{2k_j(m)-r}$.  
For this factor, we have 
\begin{align*}
v_p\left(\frac{1}{D^\bullet_{2k_j(m)-r}}\right) &\ge -1-v_p(k_j(m)-r/2)= -1-v_p(k-r/2)=:C''_r  
\end{align*} 
for sufficiently large $m$. 
Note that $v_p(k-r/2)<\infty $ because of $2k<r\le n$.

Summarizing these facts, we have some constant $C_r$ such that 
\begin{align*}
v_p(a^{(r)}_{k_j(m)}(T))\ge C_r+b(m)
\end{align*}
for all $T\in \Lambda_r^+$. 
In other words, we can find a constant $C_r$ such that 
\[v_p^{(r)}(E_{k_j(m)}^{(n)})\ge C_r+b(m). \]
This completes the proof of (1). 

\noindent
(2) Seeking a contradiction, we suppose that there exists $r$ with $0\le r<2k$ such that $v_p^{(r)}(E^{(n)}_{k_j(m)})<\nu _m$ and $v_p^{(r+1)}(E^{(n)}_{k_j(m)})=\nu _m$. 
By Theorem \ref{thm:wt_sing}, 
we should have $2k_j(m)-r\equiv 0$ mod $p-1$. 
This shows that $2k-r\equiv 0$ mod $p-1$, because of $2k_j(m)\equiv 2k$ mod $p-1$.
Since $p>2k+1$, this is impossible. 
Hence we obtain $v_p^{(r)}(E^{(n)}_{k_j(m)})\ge \nu _m$ for all $r$ with $0\le r<2k$. 
This proves (2).

\noindent 
(3) By looking at the constant term $a_{k_j(m)}(0_n)=1$, we obtain $v_p^{(0)}(E^{(n)}_{k_j(m)})=0$. 
By the property (2), we get $\nu _m \le 0$. Hence we obtain (3). 
\end{proof}

Using Theorem \ref{thm:structure}, at this stage we can prove the following. 
For $p>2k+1$ (a regular prime), $n\ge 4k$, and each $m\ge 1$, we have
\begin{align}
\label{eq:ETcong}
p^{-\nu _m}E_{k_j(m)}^{(n)}\equiv \sum _{\substack{S\in \Lambda _{2k}^+/{\rm GL}_{2k}(\Z)\\ {{\rm level}(S)\mid p^{e_m}}\\ \chi _S=\chi _p^j}} c^{(m)}_S \theta _S^{(n)} \bmod{p^{c(m)}}
\end{align}
for some $c(m)$ satisfying $c(m)\to \infty $ if $m\to \infty $.
Here, the condition on $S$ to satisfy $\chi _S=\chi_p^j$ can be obtained as follows. 
Since $a_j(m)\equiv \frac{p-1}{2^j}$ mod $p-1$, we can write as $a_j(m)=t_j(m)\cdot \frac{p-1}{2}$ for some $t_j(m)\in \Z$ such that $t_0(m)$ is even and $t_1(m)$ is odd.  
By Theorem \ref{thm:structure}, we have
\[\chi _S=\chi _p ^{t_j(m)}=\chi _p^j. \]

By the same type of arguments as in our refinement of Freitag's expansion (as in Section 7 of \cite{Bo-Ki2}), we see that 
\[c_S^{(m)}\equiv p^{-\nu _m}\cdot \frac{a_{k_j(m)}^{(2k)}(T)^*}{\epsilon (S)} \bmod{p^{c(m)}}. \]
Now we use that $a_{k_j(m)}^{(2k)}(S)^*$ depends only on the genus ${\rm gen}(S)$ of $S$ and we can rearrange the summation. 

Finally, by applying the Siegel $\Phi$-operator several times, 
we see that (\ref{eq:ETcong}) and above argument hold for all $n$.
Therefore we have the following property. 
\begin{Prop}
\label{prop:weak}
Let $n$ be a positive integer and $p$ a regular prime with $p>2k+1$.
Then we have
\begin{align}
\label{eq:eisen-theta}
p^{-\nu _m}E_{k_j(m)}^{(n)}\equiv \sum _{\substack{{\rm gen}(S)\\{\rm level}(S)\mid p^{e_m}\\ \chi _S=\chi _p^j}}a_{k_j(m)}^{(2k)}({\rm gen}(S))^*\cdot (\Theta ^{(n)}_{{\rm gen}(S)})^0 \bmod{p^{c(m)}}, 
\end{align}
where the summation in (\ref{eq:eisen-theta})
goes over finitely many genera ${\rm gen}(S)$  of $S\in \Lambda _{2k}^+/{\rm GL}_{2k}(\Z)$ with ${\rm level}(S)\mid p^{e_m}$ 
such that $\chi _S=\chi _p^j$. 
\end{Prop}
\begin{Rem}
Actually, $a_{k_j(m)}^{(2k)}({\rm gen}(S))^*$ would only depend on $\det S$ and the equivalence class of $S$ mod $p$, 
but this is possibly the same (there is only one genus of level $p$ and fixed determinant).
\end{Rem}

We will see that $a_{k_j(m)}^{(2k)}({\rm gen}(S))^*\equiv 0$ mod $p^{c(m)}$ unless ${\rm level}(S)\mid p$
in Subsection \ref{ber_part} and \ref{local_part}.
This would imply that $e_m=1$ for sufficiently large $m$.  

To prove $e_m=1$, besides checking the primitive Fourier coefficients, we also have another method. 
This is to use a property that the Siegel--Eisenstein series is a Hecke eigenfunction for the Hecke operator denoted by $T_n(p)$, which is shown in Section \ref{sec:4}.

\subsection{$p$-Adic limit of the Bernoulli part}
\label{ber_part}
The next purpose is to show that $\nu_m$ is constant for sufficiently large $m$ and the $p$-adic convergence of $a_{k_j(m)}^{(2k)}({\rm gen}(S))^*$ in (\ref{eq:eisen-theta}). 
We remark that this implies the $p$-adic convergence of $E_{k_j(m)}^{(n)}$, because of (\ref{eq:eisen-theta}). 
For this purpose, we need more detailed information on the primitive Fourier coefficients $a_{k_j(m)}^{(2k)}(T)^*$ with $T\in \Lambda _{2k}^+$.
By Theorem \ref{thm:pri}, we have
\begin{align*}
a_{k_j(m)}^{(2k)}(T)^*=b_{k_j(m)}^{(2k)}(T)\cdot c_{k_j(m)}^{(2k)}(T) 
\end{align*}
with $c_{k_j(m)}^{(2k)}(T)\in \Z$ and 
\[b_{k_j(m)}^{(2k)}(T):=(-1)^k\cdot 2^{k}\cdot \frac{k_j(m)}{B_{k_j(m)}}\cdot \frac{B_{k_j(m)-k,\eta _T}}{k_j(m)-k} \cdot \prod _{i=1}^{k}\frac{2k-2i}{B_{2k_j(m)-2i}}. \]
In this subsection, we analyze this $b_{k_j(m)}^{(2k)}(T)$.
In the following, the symbol ``$\square$'' denotes the square of an integer.
\begin{Prop}
\label{lim}
Let $p>2k+1$. 
\begin{enumerate} \setlength{\itemsep}{-3pt}
\item
If $\det (2T)\neq \square $ then 
\[\lim _{m\to \infty }b_{k_0(m)}^{(2k)}(T)=0\quad \text{and\ hence}\quad \lim _{m\to \infty }a_{k_0(m)}^{(2k)}(T)^*=0. \]

If $\det (2T)=\square$ then 
\begin{align*}
\lim _{m\to \infty }b_{k_0(m)}^{(2k)}(T)&=2^{2k}\cdot \frac{k}{(1-p^{k-1})B_{k}} \cdot \prod _{i=1}^{k-1}\frac{k-i}{(1-p^{2k-2i-1})B_{2k-2i}}. 
\end{align*}

\item
If $\det (2T)\neq p \times \square$ then 
\[\lim _{m\to \infty }b_{k_1(m)}^{(2k)}(T)=0\quad \text{and\ hence}\quad \lim _{m\to \infty }a_{k_1(m)}^{(2k)}(T)^*=0. \]
If $\det (2T)=p \times \square$ then 
\begin{align*}
\lim _{m\to \infty }b_{k_1(m)}^{(2k)}(T)&=(-1)^k\cdot 2^{2k}\cdot \frac{k}{B_{k,\chi _p}} \cdot \prod _{i=1}^{k-1}\frac{k-i}{(1-p^{2k-2i-1})B_{2k-2i}}.  
\end{align*}
\end{enumerate}
Furthermore these convergence properties are uniform in $T$. 
\end{Prop}

\begin{proof}
(1) By the same calculation as in the proof of Proposition \ref{prop:cong}, for any $i$ with $1\le i<k$ we have,  
\begin{align*}
&\lim _{m\to \infty } \frac{k_0(m)}{B_{k_0(m)}}=\frac{k}{(1-p^{k-1})B_{k}},\quad \lim _{m\to \infty } \frac{2k_0(m)-2i}{B_{2k_0(m)-2i}}=\frac{2k-2i}{(1-p^{2k-2i-1})B_{2k-2i}}. 
\end{align*}
Hence we consider the factor 
\[\frac{B_{k_0(m)-k,\eta _T}}{k_0(m)-k}\times \frac{2k_0(m)-2k}{B_{2k_0(m)-2k}}=\frac{B_{a_0(m)p^{b(m)},\eta _T}}{a_0(m)p^{b(m)}} \times \frac{2a_0(m)p^{b(m)}}{B_{2a_0(m)p^{b(m)}}}.\]
By the same argument as in Katsurada--Nagaoka \cite[page 10]{Kat-Na2}, we have  
\begin{align}
\nonumber
\label{eq:BC1}
\frac{B_{a_0(m)p^{b(m)},\eta _T}}{a_0(m)p^{b(m)}} &\times \frac{2a_0(m)p^{b(m)}}{B_{2a_0(m)p^{b(m)}}}\\ 
&\equiv 
\begin{cases}
0 \quad \text{if}\quad \eta _T\neq {\boldsymbol 1} \quad (\iff \det (2T)\neq \square)\\
2\quad \text{if}\quad \eta _T={\boldsymbol 1} \quad (\iff \det (2T)= \square)  
\end{cases}
\bmod{p^{b(m)+1}}. 
\end{align}
Here ${\boldsymbol 1}$ is the trivial character mod $1$.
This implies 
\begin{align*}
\lim_{m\to \infty }
\frac{B_{a_0(m)p^{b(m)},\eta _T}}{a_0(m)p^{b(m)}} &\times \frac{2a_0(m)p^{b(m)}}{B_{2a_0(m)p^{b(m)}}}\\
&=
\begin{cases}
0 \quad \text{if}\quad \eta _T\neq {\boldsymbol 1} \quad (\iff \det (2T)\neq \square),\\
2\quad \text{if}\quad \eta _T={\boldsymbol 1} \quad (\iff \det (2T)= \square). 
\end{cases}
\end{align*}
Summarizing above facts, if $\det (2T)\neq \square$ then
$\lim _{m\to \infty }b_{k_0(m)}^{(2k)}(T)=0$, and if $\det (2T)=\square$ then
\begin{align*}
\lim _{m\to \infty }b_{k_0(m)}^{(2k)}(T)&=2^{2k}\cdot \frac{k}{(1-p^{k-1})B_{k}} \cdot \prod _{i=1}^{k-1}\frac{k-i}{(1-p^{2k-2i-1})B_{2k-2i}}. 
\end{align*}
From (\ref{eq:BC1}), this convergence is uniform in $T$. 
 
\noindent
(2) By the same calculation as in the proof of Proposition \ref{prop:cong}, for any $i$ with $1\le i<k$, we have  
\begin{align*}
\lim _{m\to \infty } \frac{k_1(m)}{B_{k_1(m)}}=\frac{k}{B_{k,\chi _p}},\quad \lim _{m\to \infty } \frac{2k_1(m)-2i}{B_{2k_1(m)-2i}}=\frac{2k-2i}{(1-p^{2k-2i-1})B_{2k-2i}}. 
\end{align*}
We consider the factor 
\[\frac{B_{k_1(m)-k,\eta _T}}{k_1(m)-k}\times \frac{2k_1(m)-2k}{B_{2k_1(m)-2k}}=2\cdot \frac{B_{a_1(m)p^{b(m)},\eta _T}}{B_{2a_1(m)p^{b(m)}}}. \]
By the same argument as in Nagaoka \cite[page 241]{Na}, we have  
\begin{align}
\label{eq:BC2}
2\cdot \frac{B_{a_1(m)p^{b(m)},\eta _T}}{B_{2a_1(m)p^{b(m)}}}\equiv 
\begin{cases}
0 \quad \text{if}\quad \eta _T\neq \chi _p \quad (\iff \det (2T)\neq p\times \square)\\
2 \quad \text{if}\quad \eta _T=\chi _p \quad (\iff \det (2T)= p\times \square) 
\end{cases}
\bmod{p^{b(m)+1}}. 
\end{align}
This implies 
\begin{align*}
\lim _{m\to \infty } 2\cdot \frac{B_{a_1(m)p^{b(m)},\eta _T}}{B_{2a_1(m)p^{b(m)}}}=
\begin{cases}
0 \quad \text{if}\quad \eta _T\neq \chi _p \quad (\iff \det (2T)\neq p\times \square),\\
2 \quad \text{if}\quad \eta _T=\chi _p \quad (\iff \det (2T)= p\times \square). 
\end{cases}
\end{align*}
Summarizing above facts, if $\det (2T)\neq p \times \square$ then
$\lim _{m\to \infty }b_{k_j(m)}^{(2k)}(T)=0$, and if $\det (2T)=p\times \square$ then
\begin{align*}
\lim _{m\to \infty }b_{k_j(m)}^{(2k)}(T)&=(-1)^k\cdot  2^{2k}\cdot \frac{k}{B_{k,\chi _p}} \cdot \prod _{i=1}^{k-1}\frac{k-i}{(1-p^{2k-2i-1})B_{2k-2i}}. 
\end{align*}
From (\ref{eq:BC2}), this convergence is uniform in $T$. 
\end{proof}

\subsection{$p$-Adic limit of the local part}
\label{local_part}
In this subsection, we analyze the local part $c_{k_j(m)}^{(2k)}(T)$ for which $a_{k_j(m)}^{(2k)}(T)^*$ may not be zero $p$-adically. Recall from Theorem \ref{thm:pri} that $c_{k_j(m)}^{(2k)}(T)$ is given explicitly by  
\begin{align*}
c_{k_j(m)}^{(2k)}(T)= (\det (2T) f_T^{-1})^{k_j(m)-\frac{2k+1}{2}}\prod _{q\mid \det(2T)}(1-\eta _T(q)q^{k-k_j(m)})B^*_q(J^{2k_j(m)},T), 
\end{align*}
with 
\[
B^*_q(J^{2k_j(m)},T)
=\begin{cases}
&\prod _{i=1}^{\frac{s-1}{2}}(1-q^{2i+2k-2k_j(m)}) \quad \text{if}\quad s\ \text{odd},\\
&(1+\lambda _q(T)q^{\frac{2k+s-2k_j(m)}{2}})\prod_{i=1}^{\frac{s}{2}-1}(1-q^{2i+2k-2k_j(m)}) \quad \text{if}\quad s\ \text{even}. 
\end{cases}
\]
We collect some properties of $c_{k_j(m)}^{(2k)}(T)$ in the following. 
\begin{description}  \setlength{\itemsep}{-3pt}
\item[Case $j=0$.]
For $T\in \Lambda _{2k}^+$ with $\lim_{m\to \infty }a_{k_j(m)}^{(2k)}(T)\neq 0$, we have $\det (2T)
=\square $, by Proposition \ref{lim}.  
In particular, in this case $\eta _T={\boldsymbol 1}$ and $f_T=1$.  
\begin{itemize}  \setlength{\itemsep}{-3pt}
\item
If $T$ has level $p$ and $\det (2T)=p^{2d}$ with $2d=s$, then 
\[\lim _{m\to \infty }c_{k_0(m)}^{(2k)}(T)=(-1)^{\frac{s}{2}}\lambda _p(T)p^{(\frac{s}{2}-1)\frac{s}{2}}. \]
\item
If $2k\equiv 0$ mod $8$, then there exists $T_0\in \Lambda _{2k}^+$ with $\det (2T_0)=1$. 
For such $T_0$ we have $c_{k_0(m)}^{(2k)}(T_0)=1$.
\item
If $2k\equiv 4$ mod $8$, then there exists $T_2\in \Lambda _{2k}^+$ with level $p$ and $\det (2T_2)=p^2$.
We get in this case 
\begin{align*}
c_{k_0(m)}^{(2k)}(T)&=p^{2k_0(m)-2k-1}(1-p^{k-k_0(m)})(1+\lambda _p(T)p^{\frac{2k+2-2k_0(m)}{2}})\\
&=-\lambda _p(T)+p^{k_0(m)}\times \text{``some\ integer''}. 
\end{align*}
\item
If a prime $q$ divides $\det (2T)$ with $q^{2t}\; ||\; \det (2T)$ ($t\ge 1$) then $\eta _T(q)=1$
yields the factor $1-q^{k-k_0(m)}$. 
In view of $k-k_0(m)=-a_0(m)p^{b(m)}$ this gives (by Euler) a congruence
\[c_{k_0(m)}^{(2k)}(T)\equiv 0\bmod{p^{b(m)+1}}. \]
\item
If $T$ has level $p^t$ with $t\ge 2$, then $\lim _{m\to \infty } c_{k_0(m)}^{(2k)}(T)=0$.
We can confirm this as follows. 
We compare the formulas for $c_{k_0(m)}^{(2k)}(T)$ for such $T$ with $c_{k_0(m)}^{(2k)}(T')$,
where $T'$ has level $p$ and the same rank over $\F_p$ as $T$. 
The two formulas then just differ by the determinant part in front, which gives
\[c_{k_0(m)}^{(2k)}(T)=p^a c_{k_0(m)}^{(2k)}(T')\]
with $a\ge k_0(m)-\frac{2k+1}{2}$. 
This gives the claim, when combined with item 1. 
\end{itemize}

\item[Case $j=1$.] 
For $T\in \Lambda _{2k}^+$ with $\lim_{m\to \infty }a_{k_1(m)}^{(2k)}(T)\neq 0$, we have $\det (2T)
=p\times \square $, by Proposition \ref{lim}. 
In this case, we have $\eta _T=\chi _p$, $f_T=p$.  
\begin{itemize}  \setlength{\itemsep}{-3pt}
\item
If $T$ has level $p$ and $\det (2T)=p^{1+2d}$ with $1+2d=s$, then 
\[\lim _{m\to \infty }c_{k_1(m)}^{(2k)}(T)=(-1)^{\frac{s-1}{2}}p^{(\frac{s-1}{2})^2}. \]
\item
There exists $T_1\in \Lambda _{2k}^+$ with level $p$ and $\det (2T_1)=p$. 
For such $T_1$ we have $c_{k_1(m)}^{(2k)}(T_1)=1$.

\item
If a prime $q$ divides $\det (2T)$ with $q^{2t}\; ||\; \det (2T)$ ($t\ge 1$) then 
we have 
\[1-\eta _T(q)q^{k-k_1(m)}=1-\chi_p (q)q^{-a_1(m)p^{b(m)}}\equiv 0 \bmod{p^{b(m)+1}}\] 
because of $q^{-a_1(m)p^{b(m)}}\equiv \chi_p(q)$ mod $p^{b(m)+1}$. 
This implies that 
\[c_{k_1(m)}^{(2k)}(T)\equiv 0\bmod{p^{b(m)+1}}. \]
\item
If $T$ has level $p^t$ with $t\ge 2$, then $\lim _{m\to \infty } c_{k_1(m)}^{(2k)}(T)=0$.
The reason is the same as item 5 for the case of $j=0$. 
\end{itemize}
\end{description}

As a conclusion of these observations, we have the following statements: 
We put 
\begin{align*}
\mu_0:=v_p\left(2^{2k}\cdot \frac{k}{(1-p^{k-1})B_{k}} \cdot \prod _{i=1}^{k-1}\frac{k-i}{(1-p^{2k-2i-1})B_{2k-2i}}\right),\\
\mu_1:=v_p\left((-1)^k\cdot 2^{2k}\cdot \frac{k}{B_{k,\chi _p}} \cdot \prod _{i=1}^{k-1}\frac{k-i}{(1-p^{2k-2i-1})B_{2k-2i}}\right). 
\end{align*}
For any $T\in \Lambda _{2k}^+$ and for large $m$, we claim that
\[v_p(a_{k_j(m)}^{(2k)}(T)^*)\ge \mu _j\]
since $a_{k_j(m)}^{(2k)}(T)^*=b_{k_j(m)}^{(2k)}(T)\cdot c_{k_j(m)}^{(2k)}(T)$ with $c_{k_j(m)}^{(2k)}(T)\in \Z$, 
and by Proposition \ref{lim}.  
Then we have also $v_p(a_{k_j(m)}^{(2k)}(T))\ge \mu _j$ because of (\ref{eq:rel_pri}). 
Furthermore, we have
\begin{align*}
&v_p(a^{(2k)}_{k_0(m)}(T_0))=\mu _0 \quad \text{if}\quad k\equiv 0 \bmod{4}, \\
&v_p(a^{(2k)}_{k_0(m)}(T_2))=\mu _0 \quad \text{if}\quad k\equiv 2 \bmod{4}, \\
&v_p(a^{(2k)}_{k_1(m)}(T_1))=\mu _1, 
\end{align*}
for $T_l\in \Lambda _{2k}^+$ ($l=0$, $1$, $2$) with $\det (2T_l)=p^l$. 
(The existence of them in each cases has already been discussed above.) 
This is because in these cases the primitive Fourier coefficient and the Fourier coefficient coincide.
Therefore we have $v_p^{(2k)}(E_{k_j(m)}^{(n)})=\mu _j$ for sufficiently large $m$. 

The above discussion yields the following assertion, which is a more precise version of Theorem \ref{thm:main}.
\begin{Thm}
\label{thm:coin}
Let $n$ be positive integer and $p$ a regular prime with $p>2k+1$.
\begin{enumerate} \setlength{\itemsep}{-3pt}
\item
Let $n\ge 2k$. Then we have $(\nu_m=)$ $v_p^{(2k)}(E_{k_j(m)}^{(n)})=\mu _j$ for sufficiently large $m$. 
Namely we have $\nu _m \to \mu _j$ ($m\to \infty$) for each $j$. 
\item
For sufficiently large $m$, each $p^{-\mu _j}E_{k_j(m)}^{(n)}\in M_{k_j(m)}(\Gamma _n)_{\Z_{(p)}}$ ($j=0$, $1$) is mod $p^{c(m)}$ singular of $p$-rank $2k$ and we have
\begin{align*}
p^{-\mu _j}E_{k_j(m)}^{(n)}\equiv 
\sum _{\substack{{\rm gen}(S)\\{\rm level}(S)\mid p\\ \chi _S=\chi _p^j}}
p^{-\mu _j}a_{k_j(m)}^{(2k)}({\rm gen}(S))^*\cdot (\Theta ^{(n)}_{{\rm gen}(S)})^0 \bmod{p^{c(m)}}, 
\end{align*}
where $c(m)$ is a certain positive integer satisfying $c(m)\to \infty $ if $m\to \infty$. 
Here the summation goes over finitely many genera ${\rm gen}(S)$ of $S\in \Lambda_{2k}^+/{\rm GL}_{2k}(\Z)$ such that ${\rm level}(S)\mid p$ and $\chi _S=\chi _p^j$. 
\item
The sequence $E_{k_j(m)}^{(n)}$ converges $p$-adically and we have
\begin{align}
\label{eq:sum}
\lim_{m\to \infty }E_{k_j(m)}^{(n)}= \sum _{\substack{{\rm gen}(S)\\ {\rm level}(S)\mid p\\ \chi _S=\chi _p^j}}a_j({\rm gen}(S))\cdot (\Theta ^{(n)}_{{\rm gen}(S)})^0, 
\end{align}
where 
\[a_j({\rm gen}(S)):=\lim _{m\to \infty }a_{k_j(m)}^{(2k)}({\rm gen}(S))^*\in \Q. \]
More explicitly, for $S\in \Lambda _{2k}^+$ with ${\rm level}(S)\mid p$, $\det (2S)=p^{s}$ and
rank $2k-s$ over $\F_p$, we have
\begin{align*}
&a_0({\rm gen}(S))\\
&~~~~~=2^{2k}\cdot \frac{k}{(1-p^{k-1})B_{k}} \cdot \prod _{i=1}^{k-1}\frac{k-i}{(1-p^{2k-2i-1})B_{2k-2i}} 
\cdot
(-1)^{\frac{s}{2}}\lambda _p(T)p^{(\frac{s}{2}-1)\frac{s}{2}},\\
&a_1({\rm gen}(S))\\
&~~~~~=(-1)^k\cdot 2^{2k}\cdot \frac{k}{B_{k,\chi _p}} \cdot \prod _{i=1}^{k-1}\frac{k-i}{(1-p^{2k-2i-1})B_{2k-2i}}  
\cdot (-1)^{\frac{s-1}{2}}p^{(\frac{s-1}{2})^2}.
\end{align*} 
In particular, the $p$-adic Siegel--Eisenstein series $\widetilde{E}_{{\boldsymbol k}_j}^{(n)}$ does not depend on the 
choice of $k_j(m)$ satisfying $k_j(m)\to {\boldsymbol k}_j$ ($m\to \infty$) in ${\boldsymbol X}$. 
\end{enumerate}
\end{Thm}
\begin{Rem}
\begin{enumerate} \setlength{\itemsep}{-3pt}
\item
Concerning the regularity condition on $p$, we note that it is not necessary for the property (1); 
for properties the actual condition we need is Assumption \ref{Assm:NV}. 
The regularity condition on $p$ is not necessary for the property (1). 
For the regularity condition on $p$ in (2), (3), the actual condition we need is Assumption \ref{Assm:NV}. 
\item
For any $S\in \Lambda_{2k}^+$ with ${\rm level}(S)\mid p$, there exists $U\in {\rm GL}_{2k}(\Z_p)$ such that
\[S[U]={\rm diag}(1,\cdots ,1 , \underbrace{p,\cdots , p}_{s})\quad \text{or} \quad {\rm diag}(1,\cdots ,1 , \gamma , \underbrace{p,\cdots , p, p\gamma '}_{s}),\] 
where $\gamma $, $\gamma ' \in \Z_p^{\times }-(\Z_p^{\times })^2$, and $(\Z_p^{\times })^2$ is the set of all square elements of $\Z_p^{\times }$. 
In particular we have $0\le s\le 2k$ for $s$ with $\det (2S)=p^s$. 
Also, one should observe that NOT all cases arise, e.g. $2k=4$, $s=0$ does not exist. 
\end{enumerate}
\end{Rem}

\section{Characterization of $\widetilde{E}_{{\boldsymbol k}_j}^{(n)}$ by $U(p)$-operator}
\label{sec:3.5} 
The goal of this section is to prove Corollary \ref{cor:eisen} (2).

Let $F\in M_k(\Gamma _n)_{\Z_{(p)}}$. 
Let $T_n(p)$ be the Hecke operator defined as 
\begin{align}
\label{eq:hecke}
F\mid T_n(p):=p^{nk-\frac{1}{2}n(n+1)}\sum _{\smat{A}{B}{0_n}{D}\in \Gamma_n \backslash \Gamma _n\smat{1_n}{0_n}{0_n}{p1_n}\Gamma _n}\det D^{-k}\cdot F((AZ+B)D^{-1}). 
\end{align}
We define the operator $U(p)$ as usual by
\begin{align}
\label{eq:Up}
F=\sum _Ta_F(T)q^T\longmapsto F\mid  U(p)=\sum _Ta_F(pT)q^T. 
\end{align} 

\begin{Lem}
\label{lem:eigen_Tp}
Let $n$, $k$ be positive integers with $k\ge n$ and $p$ a prime. 
Then we can write $T_n(p):M_k(\Gamma _n)\to M_k(\Gamma _n)$ as the form $T_n(p)=U(p)+p^{k-n}\widetilde{V}(p)$ 
with an operator $\widetilde{V}(p)$, which is some linear combination of operators preserving integrality of the Fourier coefficients. 
\end{Lem}
\begin{proof}
Note that $A{}^tD=p1_n$ for $A$, $D$ in the index of summation in (\ref{eq:hecke}). 
We separate the summation into two parts; the one is of $A=1_n$ and hence $D=p1_n$, the other is of $A\neq 1_n$ and hence $\det A=p^i$ for some $i$ with $1\le i\le n$.  
Then we have 
\begin{align*}
F\mid T_n(p)&:=p^{nk-\frac{1}{2}n(n+1)}\sum _{\smat{1_n}{B}{0_n}{p\cdot 1_n}}p^{-nk}F\left(\frac{Z+B}{p}\right) \\
&+ p^{nk-\frac{n(n+1)}{2}}\sum _{\smat{A}{B}{0_n}{D_A}}\det D_A^{-k}\cdot F((AZ+B)D_A^{-1})
\end{align*}
with $D_A:=p\cdot {}^tA^{-1}$. 

The first part is just $F\mid U(p)$, namely we have 
\[F\mid U(p)=p^{nk-\frac{1}{2}n(n+1)}\sum _{\smat{1_n}{B}{0_n}{p
1_n}}p^{-nk}F\left(\frac{Z+B}{p}\right). \]
The second part becomes  
\begin{align}
\nonumber
&p^{nk-\frac{n(n+1)}{2}}\sum _{\smat{A}{B}{0_n}{D_A}}\det D_A^{-k}\cdot F((AZ+B)D_A^{-1})\\ 
\label{eq:2}
&=p^{nk-\frac{n(n+1)}{2}}\sum _{\substack{1\le i \le n}} \sum _{\substack{a_1,\cdots, a_n \\ a_j \mid a_{j-1} \\ \prod a_j =p^i}}\sum _{A\sim {\rm diag}(a_1,\cdots,a_n)}\det D_A^{-k}\cdot F((AZ+B)D_A^{-1}), 
\end{align}
where 
$A$ in the third summation of (\ref{eq:2}) runs over matrices such that $U A V={\rm diag}(a_1,\cdots,a_n)$ for some $U$, $V\in {\rm GL}_n(\Z)$ (we denoted this situation as $A\sim {\rm diag}(a_1,\cdots,a_n)$). 

Mizumoto \cite{Mizu} (before Remark 1.3) proved that
$\prod _{j=1}^na_j^{k-n+j-1}$ (a power of $p$) divides the Fourier expansion of the part (\ref{eq:2}).  
Since $p\mid a_1$, we have $v_p\left(\prod _{j=1}^na_j^{k-n+j-1}\right)\ge k-n$. 
This implies $v_p((\ref{eq:2}))\ge k-n$. 
This completes the proof. 
\end{proof}
\begin{Lem}
\label{lem:unit_eigen}  
Let $n$ be a positive integer, $k$ a positive even integer with $k>n+1$, and $p$ a prime. 
Let $\lambda _{n,k}(p)$ be the eigenvalue of $E^{(n)}_{k}$ for the Hecke operator $T_n(p)$. 
Then we have 
\[\lambda _{n,k}(p)=p^{nk-\frac{n(n+1)}{2}}+\lambda _{n-1,k}(p), \]  
and in particular $\lambda _{n,k}(p) \in \Z_{(p)}^\times $.  
\end{Lem}
\begin{proof}
  Let ${\mathcal T}_n(p)$ be the Hecke operator treated as
  Freitag's book \cite{Frei2}. 
Let $F\in M_k(\Gamma _n)$. 
Comparing the definitions, we have $T_n(p)=p^{nk-\frac{n(n+1)}{2}}{\mathcal T}_{n}(p)$. 
On the other hand, Freitag \cite{Frei2} and \v{Z}arkoskaya \cite{Za} proved that 
\[\Phi (F\mid {\mathcal T}_n(p))=\Phi(F)+p^{n-k}\Phi (F)\mid {\mathcal T}_{n-1}(p). \]
Multiplying by $p^{nk-\frac{n(n+1)}{2}}$, we get 
\[\Phi (F\mid T_n(p))=p^{nk-\frac{n(n+1)}{2}}\Phi (F)+\Phi (F)\mid T_{n-1}(p). \]
Applying this formula to $F=E_k^{(n)}$, we have 
\[\lambda _{n,k}(p)=p^{nk-\frac{n(n+1)}{2}}+\lambda _{n-1,k}(p). \] 
Since $\lambda _{1,k}(p)\in \Z_{(p)}^\times $, we get $\lambda _{n,k}(p)\in \Z_{(p)}^\times $ for all $n$ with $n\ge 1$. 
\end{proof}

\begin{proof}[Proof of Corollary \ref{cor:eisen} (2)]
It follows from Lemma \ref{lem:eigen_Tp} 
that 
\[p^{-\mu _j} E_{k_j(m)}^{(n)} \mid U(p) \equiv p^{-\mu _j} E_{k_j(m)}^{(n)} \mid T_{n}(p)=\lambda _{n,k_j(m)}(p)\cdot p^{-\mu _j} E_{k_j(m)}^{(n)} \bmod{p^{c(m)}} \] 
with some $c(m)$ satisfying $c(m)\to \infty $ if $m\to \infty $. 
On the other hand, Lemma \ref{lem:unit_eigen} 
implies that 
\[\lambda _{n,k_j(m)}(p)\equiv \lambda _{1,k_j(m)}(p)=\sum _{0<d\mid p}d^{k_j(m)-1}\equiv 1 \bmod{p^{c'(m)}}\] 
with some $c'(m)$ satisfying $c'(m)\to \infty $ if $m\to \infty $. 
These imply that 
\[p^{-\mu _j} E_{k_j(m)}^{(n)} \mid U(p) \equiv p^{-\mu _j} E_{k_j(m)}^{(n)}  \bmod{p^{\min\{c(m),c'(m)\}}}. \]
Taking a $p$-adic limit, we obtain  
$\widetilde{E}_{{\boldsymbol k}_j}^{(n)}\mid U(p)=\widetilde{E}_{{\boldsymbol k}_j}^{(n)}$. 
\end{proof}

\section{Another way to show that level $p$ suffices}
\label{sec:4} 

In Subsection \ref{local_part}, we proved that $e_m$ in Proposition \ref{prop:weak} is $1$ for large $m$ by checking directly
$a_{k_j(m)}^{(2k)}({\rm gen}(S))^*\equiv 0$ mod $p^{c(m)}$ for $p^2\mid {\rm level}(S)$. 
In this section, we introduce another way to prove this. 

\begin{Lem}
\label{lem:U_p}
Let $r$ be a positive even integer and $p$ a prime. 
Let $S\in \Lambda _r^+$ be of level $p^e$. 
Then for any $n$, 
$\theta _S^{(n)}\mid U(p)$ is an integral linear combination of theta series $\theta _{S'}^{(n)}$ with $S'\in \Lambda _r^+$ satisfying $\chi _{S'}=\chi _S$ and ${\rm level}(S') \mid \max\{p,p^{e-1}\}$.  
\end{Lem}
One may extract such statement from Evdokimov \cite{Evd} or
from
\cite{Bo-Fu-Sch}, but 
we prefer to give a simple-minded direct proof here. 
Our proof also works for arbitrary levels (not only powers of $p$).
\begin{proof}
Let $\{S_1,\cdots ,S_h\}$ be a set of representatives for ${\rm GL}_r(\Z)$-equivalence classes
in $\Lambda _r^+$ of level dividing $p^e$ with nebentypus equal to $\chi _S$. 
Then 
\[\theta^{(n)} _{S_i}\mid U(p)=\sum_j \alpha (S_i,S_j)\theta^{(n)} _{S_j} \]
holds with certain coefficients $\alpha (S_i,S_j)\in \C$. 
Moreover, by general properties of $U(p)$, we have $\alpha (S_i,S_j)=0$ unless ${\rm level}(S_j)\mid \max\{p,p^{e-1}\}$. 
This follows from the theory of singular modular forms for $n$ large enough and then for arbitrary $n$
by the commutation relation for $U(p)$ and the Siegel $\Phi $-operator. 
We analyze the coefficients $\alpha (S_i,S_j)$ for the case $n=r$. 
For all $i$, $j$, we have 
\[\frac{A(S_i,pS_j)}{\epsilon(S_j)}=\sum _t \alpha (S_i,S_t)\cdot \frac{A(S_t,S_j)}{\epsilon (S_j)}.\] 
We consider this as a system of linear equations for $\alpha (S_i,S_j)$.
The coefficient matrix $\left(\frac{A(S_t,S_j)}{\epsilon (S_j)}\right)_{i,j}$ is upper triangular with integral entries and with diagonal entries equal to one, 
if we choose $S_i$ such that $\det(S_{i+1})\ge \det(S_i)$ for all $i$. 
Then we have $\alpha (S_i,S_j)\in \Z$. 
\end{proof}
\begin{Rem}
We tacitly used the fact that $U(p)$ maps modular forms
for $\Gamma_0^{(n)}(p^e)$ to modular forms for $\Gamma^{(n)}_0(p^{e-1})$ provided that
  $e\geq 2$ and the nebentypus is defined mod $p^{e-1}$.
  This follows in a straightforward way from the equality of double cosets
  $$\Gamma_0^{(n)}(p^e)\cdot\left(\begin{array}{cc} 1_n & 0_n \\
    0_n & p1_n\end{array}\right)\cdot \Gamma_0^{(n)}(p^e)=
\Gamma_0^{(n)}(p^e)\cdot\left(\begin{array}{cc} 1_n & 0_n\\
    0_n & p1_n\end{array}\right)\cdot \Gamma_0^{(n)}(p^{e-1})\qquad (e\geq 2)$$
or by observing that the proof of Lemma 1 in 
Li \cite{Li} also works for degree $n$. 
\end{Rem}
\begin{Thm}
\label{Thm:prec_lev}
Let $n$ be a positive integer, $k$ a positive even integer, and $p$ a prime. 
Suppose that $F\in M_k(\Gamma _n)_{\Z_{(p)}}$ satisfies $F\mid T_n(p)\equiv \epsilon F$ mod $p^m$ for some $\epsilon \in \Z_{(p)}^{\times}$.  
Suppose also that, for some $r\le n$, we have a congruence
\[F\equiv \sum _{\substack{S\in \Lambda _r^+/{\rm GL}_r(\Z)\\ {\rm level}(S)\mid p^{e_m}}} c_S \theta _S^{(n)} \bmod{p^m}.\] 
Then we have $e_m=1$. In other words, we get
\[F\equiv \sum _{\substack{S\in \Lambda _r^+/{\rm GL}_r(\Z)\\ {\rm level}(S)\mid p}} c_S \theta _S^{(n)} \bmod{p^m}.\] 
\end{Thm}
\begin{proof}
By Lemma \ref{lem:U_p},
$\theta _S^{(n)}\mid U(p)$ is integral linear combination of $\theta ^{(n)}_{T}$ with ${\rm level}(T)\mid {\rm level}(S)$ but ${\rm level}(T)\neq {\rm level}(S)$ if $p^2 \mid {\rm level}(S)$. 
Then for some $N$, we have 
\[\epsilon ^N F\equiv F\mid T_n(p)^N\equiv \sum _{\substack{S\in \Lambda _r^+/{\rm GL}_r(\Z)\\ {\rm level}(S)\mid p}} d_S \theta _S^{(n)} \bmod{p^m} \]
with $\epsilon ^N \in \Z_{(p)}^\times$. Therefore we obtain the assertion. 
\end{proof}
By Lemma \ref{lem:unit_eigen}, we can apply Theorem \ref{Thm:prec_lev} as $F=p^{-\nu _m}E_{k_j(m)}^{(n)}$ in the situation of Proposition \ref{prop:weak}. Hence $e_m$ in Proposition \ref{prop:weak} is $1$.
\section*{Acknowledgment}
The authors would like to thank Professors K.~Gunji and R.~Schulze-Pillot for their kind information about the multiplicity for $U(p)$-eigenvalue $1$ and Siegel's main theorem, which are described in Remark \ref{rem:multi} (3). 
The authors would like to thank also the referee for valuable comments. This work was supported by JSPS KAKENHI Grant Number 22K03259.


\section*{Conflict of interest statement}
On behalf of all authors, the corresponding author states that there is no conflict of interest.

\providecommand{\bysame}{\leavevmode\hbox to3em{\hrulefill}\thinspace}
\providecommand{\MR}{\relax\ifhmode\unskip\space\fi MR }
\providecommand{\MRhref}[2]{%
  \href{http://www.ams.org/mathscinet-getitem?mr=#1}{#2}
}
\providecommand{\href}[2]{#2}

\begin{flushleft}
Siegfried B\"ocherer\\
Kunzenhof 4B \\
79177 Freiburg, Germany \\
Email: boecherer@t-online.de
\end{flushleft}

\begin{flushleft}
  Toshiyuki Kikuta\\
  Faculty of Information Engineering\\
  Department of Information and Systems Engineering\\
  Fukuoka Institute of Technology\\
  3-30-1 Wajiro-higashi, Higashi-ku, Fukuoka 811-0295, Japan\\
  E-mail: kikuta@fit.ac.jp
\end{flushleft}

\end{document}